\documentclass[a4paper]{amsart}
\usepackage[latin9]{inputenc}
\usepackage{a4}
\usepackage{amsfonts}
\usepackage{amssymb}
\usepackage{amsmath}
\usepackage{amsthm}
\usepackage{changepage}
\usepackage{nomencl}
\usepackage{geometry}
\usepackage{tikz}
\usepackage{verbatim}
\usetikzlibrary{matrix}
\usepackage{hyperref}
\usepackage[toc,page]{appendix}
\usepackage{mathrsfs}
\usepackage[square, numbers, comma]{natbib}

\usepackage{enumerate}
\usepackage{amsbsy}

\usepackage[all]{xy}
\usepackage{pb-diagram,pb-xy}
   \bibpunct{[}{]}{,}{n}{}{,}
\input cyracc.def

\begin{document}
\providecommand{\keywords}[1]{\textbf{\textit{Keywords: }} #1}
\newtheorem{theorem}{Theorem}[section]
\newtheorem{lemma}[theorem]{Lemma}
\newtheorem{proposition}[theorem]{Proposition}
\newtheorem{corollary}[theorem]{Corollary}
\newtheorem{problem}[theorem]{Problem}
\newtheorem{question}[theorem]{Question}
\newtheorem{conjecture}[theorem]{Conjecture}
\newtheorem{claim}[theorem]{Claim}
\newtheorem{condition}[theorem]{Condition}

\theoremstyle{definition}
\newtheorem{definition}[theorem]{Definition} 
\theoremstyle{remark}
\newtheorem{remark}[theorem]{Remark}
\newtheorem{example}[theorem]{Example}
\newtheorem{condenum}{Condition}

\def\p{\mathfrak{p}}
\def\q{\mathfrak{q}}
\def\s{\mathfrak{S}}
\def\Gal{\mathrm{Gal}}
\def\Ker{\mathrm{Ker}}
\def\soc{\mathrm{soc}}
\def\Coker{\mathrm{Coker}}
\newcommand{\cc}{{\mathbb{C}}}   
\newcommand{\ff}{{\mathbb{F}}}  
\newcommand{\nn}{{\mathbb{N}}}   
\newcommand{\qq}{{\mathbb{Q}}}  
\newcommand{\rr}{{\mathbb{R}}}   
\newcommand{\zz}{{\mathbb{Z}}}  
\def\K{\kappa}

\title{Sparsity of stable primes for dynamical sequences}
\author{Joachim K\"onig}
\address{Department of Mathematics Education, Korea National University of Education, Cheongju 28173, South Korea}
\email{jkoenig@knue.ac.kr}

\maketitle
\begin{center}
{\it Dedicated to Peter M\"uller on the occasion of his 60th birthday}
\end{center}
\begin{abstract}
We show that a dynamical sequence $(f_n)_{n\in \mathbb{N}}$ of polynomials over a number field whose set of stable primes is of positive density must necessarily have a very restricted, and in particular virtually pro-solvable dynamical Galois group. Together with existing heuristics, our results suggest moreover that a polynomial $f$ all of whose iterates are irreducible modulo a positive density subset of the primes must necessarily be a composition of linear functions, monomials and Dickson polynomials.
\end{abstract}

\section{Introduction and main results}

A central question in arithmetic dynamics is the following: given a nonlinear polynomial $f$ with rational coefficients, can one find (many) values $a\in \mathbb{Q}$ such that $f^{\circ n}(X)-a$ is irreducible for all $n$, where $f^{\circ n}$ denotes the $n$-th iterate of $f$? In this case, the pair $(f,a)$ is called stable. The above question, which goes back to Odoni \cite{Odoni}, can be seen as the quest for a profinite analogue of Hilbert's irreducibility theorem (which ensures the existence of infinitely many $a$ yielding irreducibility for a {\it fixed} $n$). See also \cite{Jones} or \cite{JonesLevy} for investigations of this and related notions. 
Stability has furthermore been investigated modulo primes, and a key question is then, given a pair $(f,a)$ as above, how large is the set of primes $p$ such that $f^{\circ n}(X)-a$ is irreducible modulo $p$ for all $n$? The same questions can of course be asked for composition of not necessarily identical polynomials. We thus recall the following definitions.

\begin{definition}
Let $\kappa$ be a number field, and $(f_n)_{n\in \mathbb{N}}$ an infinite sequence of  polynomials in $\kappa[X]$ of degree $\ge 2$. A prime $p$ of $\kappa$ is called a stable prime for $(f_n)$ if, for all $n\in \mathbb{N}$, the mod-$p$ reduction of $f_1\circ \dots \circ f_n$ is defined, of non-decreasing degree, and irreducible. In particular, for $f\in \kappa[X]$ and $a\in \kappa$, $p$ is called a stable prime for $(f,a)$ iff it is a stable prime for the sequence $(f-a,f,f,\dots)$, i.e., if ($p$ does not divide the leading coefficient of $f$, and) all the polynomials $f^{\circ n}(X)-a$, $n\in \mathbb{N}$, are (defined and) irreducible modulo $p$.
\end{definition}

While there are known examples of pairs $(f,a)$ whose set of stable primes is of positive density, such as $f=x^2-2$, $a=0$ (see \cite[Remark 2.5]{Ferr}), it is generally expected that these occur only in quite exceptional situations. 
The objective of this article is to give a strong necessary (group-theoretical) condition for a dynamical sequence $(f_n)_{n\in \mathbb{N}}$ to have a positive density set of stable primes, which may on the one hand be used to conveniently derive ``density zero" conclusions in many individual cases, and which on the other hand gives the notion of exceptionality in the context of mod-$p$ stability a much more concrete shape, something that will be expanded on below. For a sequence $\mathcal{S}:=(f_n)_{n\in \mathbb{N}}$ of polynomials over a field $\K$, let us denote by $G_{\mathcal{S},\K}$ the projective limit of $\textrm{Gal}(f_1\circ \dots \circ f_n/\K)$, $n\to \infty$.
 
\begin{theorem}
\label{thm:main}
Let $\mathcal{S}:=(f_n)_{n\in \mathbb{N}}$ a sequence of polynomials over a number field $\K$, and assume that the set of stable primes of $\mathcal{S}$ is of positive density. Then 
there exists a finite extension $L/\K$ such that for all $n\in \mathbb{N}$, $\textrm{Gal}(f_1\circ\dots\circ f_n/L)$ embeds (as a permutation group) into an iterated wreath product $AGL_1(p_{k_n})\wr\dots\wr AGL_1(p_{1}) (\le S_{p_1\cdots p_{k_n}})$ for some index $k_n\in \mathbb{N}$ and prime numbers $p_1,\dots, p_{k_n}$. In particular, the projective limit $G_{\mathcal{S}, L}$ is pro-solvable.
\end{theorem}

Theorem \ref{thm:main} makes no assertion about the density of stable primes in the case where the monodromy groups $\textrm{Mon}_{f_i}:=\textrm{Gal}(f_i(X)-t/\K(t))$ already embed into groups $AGL_1(p_i)$. A ``density zero" conclusion incorporating this case as well was reached in Theorem 1.4(1) of \cite{KNR},\footnote{In fact, that result achieves a conclusion for a somewhat more general notion of stability.} under a ``large kernel" assumption, strengthening an earlier result of Ferraguti \cite{Ferr}, which reached the same conclusion under the assumption that the index of $G_n:=\textrm{Gal}(f_1\circ\dots\circ f_n/\K)$ inside the maximal possible group, namely the full iterated wreath product $S_{d_n}\wr\dots\wr S_{d_1}$ (where $d_i:=\deg(f_i)$), is $o(d_1\cdots d_n)$ as $n\to \infty$.
On the other hand, away from the affine case, Theorem \ref{thm:main} constitutes a significant improvement, yielding that the proportion of stable primes of $(f,a)$ is $0$ as soon as the order of $G_n$ grows asymptotically faster than the order of the $n$-fold iterated wreath product $U_n\wr\dots\wr U_1$ for any solvable subgroups $U_i\le S_{d_i}$.
In the special case where all $f_i$ are of bounded degree, comprising in particular the important case of iteration of one polynomial, we give a slight strengthening in Theorem \ref{thm:main_add}, which yields the density zero conclusion as long as the order of $G_n$ grows asymptotically faster than that of any iterated wreath product of {\it cyclic} groups (inside the symmetric group of degree $\deg(f_1)\cdots\deg(f_n)$). To see how dramatically small these groups are compared to the full iterated wreath product $S_{d_n}\wr\dots\wr S_{d_1}$, consider the case $d_i=p$ a prime number, in which case the above conclusion works with $U_1=\dots=U_n=C_p$. 

The proof of Theorem \ref{thm:main} is group-theoretic and will be achieved via a precise bound on the number of $n$-cycles in transitive permutation groups $G\le S_n$, see Theorem \ref{thm:ncycleprop}.

Regarding the role of the groups $AGL_1(p_i)$ in Theorem \ref{thm:main}, recall that the degree-$n$ Dickson polynomial with parameter $\alpha$ is defined as the unique degree-$n$ polynomial $D_{n,\alpha}$ satisfying $D_{n,\alpha}(X+\frac{\alpha}{X}) = X^n+ (\frac{\alpha}{X})^n$.\footnote{Dickson polynomials are linearly related (i.e., equal up to left- and right-composition with two linear polynomials $\lambda,\mu$) over $\mathbb{C}$, but not necessarily over $\K$ to the $n$-th Chebyshev polynomial.} In the context of monodromy, the groups $AGL_1(p)$ in the assertion of Theorem \ref{thm:main} gain a special relevance, as they are well-known to occur as the (arithmetic, over $\mathbb{Q}$) monodromy groups of monomials $X^p$ and Dickson polynomials $D_{p,\alpha}$, and ``essentially" only of those, see Theorem \ref{thm:monclass}.
In the context of Theorem \ref{thm:main}, this means, vaguely speaking, that for a polynomial $f$ which is {\it not} composed of polynomials linearly related to monomials or Dickson polynomials, the only way to obtain a positive density set of stable primes is for the groups $\textrm{Gal}(f^{\circ n}(X)-a/\K)$ ($n\in \mathbb{N}$) to be of much slower growing order than their geometric counterparts $\textrm{Gal}(f^{\circ n}(X)-t/\K(t))$. For example, in the case where $f$ is an indecomposable polynomial with monodromy group $A_d$ or $S_d$ ($d\ge 5$), it is shown in \cite{KNR} that the monodromy group $\textrm{Gal}(f^{\circ n}(X)-t/\K(t))$ of iterates $f^{\circ n}$ contains the full iterated wreath product $A_d\wr\dots\wr A_d$ for all $n\in \mathbb{N}$. It is furthermore expected that the projective limit $G_{a,\infty}:=\varprojlim \textrm{Gal}(f^{\circ n}(X)-a/\K)$ is ``usually" (i.e., for reasonably general $f$ and for all $a$ outside of some very concrete exceptions) very large inside the maximal possible group $G_{t,\infty}:=\varprojlim\textrm{Gal}(f^{\circ n}(X)-t/\K(t))$. For example, under rather general assumptions, $G_{a,\infty}$ is expected to be of finite index inside $G_{t,\infty}$, although this is hard to prove in general. See, e.g., \cite{BDGHT} or \cite{Jones13}.

 In view of such considerations, combined with Theorem \ref{thm:main}, it seems reasonable to ask:

\begin{question}
\label{ques:main}
Let $f\in \K[X]$ be a polynomial which is not a composition of polynomials linearly related to monomials or Dickson polynomials. 
Is it true that the proportion of stable primes for $(f,a)$ is $0$ for all $a\in \K$?
\end{question}

We give some evidence in favor of a positive answer to Question \ref{ques:main} in Theorem \ref{thm:evidence}. 
Answering the question in full seems very hard at this point. Special cases have been treated, e.g., in \cite{MOS} (for $f$ a trinomial, and $a=0$) and \cite{KNR} (for certain PCF maps and infinitely many $a$). Theorem \ref{thm:main} should allow to obtain such results for wider classes of polynomials, but this will be the subject of future research. 
We note that, in view of Theorem \ref{thm:main}, Question \ref{ques:main} asks about (indeed, very particular) solvable backward orbits $\cup_{n\in \mathbb{N}}f^{-n}(a)$ of polynomials $f$. 
Comparatively better understood is the even more special case of abelian backward orbits, cf.\ \cite{DZ} or \cite{FOZ}.

At this point, we at least have the following corollary about dynamical sequences having a {\it uniformly} lower-bounded set of stable primes for ``many" different fibers.
\begin{corollary}
\label{cor:manyfibers}
Let $(f_n)_{n\in \mathbb{N}}$ be a sequence of indecomposable polynomials $f_n\in \K[X]$, and $\epsilon>0$. If the proportion of stable primes of the sequence $(f_1(X)-a, f_2, f_3,\dots)$ exceeds $\epsilon$ for a non-thin set\footnote{Recall that a subset of a field $\K$ is called ``thin" if it is contained in the union of finitely many value sets of nontrivial morphisms $\varphi_i:X_i\to \mathbb{P}^1_\K$ of algebraic curves $X_i$ over $\K$.} of values $a\in \K$, then $f_n$ is linearly related over $\K$ to a monomial or a Dickson polynomial, for all but finitely many $n\in \mathbb{N}$. In particular, if $f\in \K[X]$ is any (not necessarily indecomposable) polynomial for which the proportion of stable primes of $(f,a)$ exceeds $\epsilon$ for a non-thin set of $a\in \K$, then $f$ is a composition of polynomials all linearly related to monomials or Dickson polynomials. \end{corollary}

Note finally that the case when $f$ is a composition of linears, monomials and/or Dickson polynomials is very different, and there are indeed several known situations where the set of stable primes of $(f,a)$ is of positive density, see, e.g., \cite[Section 1.2]{MOS} or \cite[Remark 2.5]{Ferr}. Cf.\ also \cite[Section 19]{Benedetto} for a short overview of problems related to stability of primes.


{\bf Acknowledgement}: Investigations in the direction of Theorem \ref{thm:ncycleprop} were sparked a long time ago due to suggestions by Peter M\"uller. 
More recent discussions with Danny Neftin made the author aware of the relevance to arithmetic dynamics. I also thank the anonymous referee for several valuable suggestions.

\section{Preliminaries}

\subsection{Reduction to group theory} 
Claims about stable primes such as the one in Theorem \ref{thm:main} reduce directly to claims about permutation groups, due to the following well-known translations.
Firstly, in order for the polynomial $f_1\circ\dots\circ f_n$ to be irreducible modulo $p$, it of course needs to be separable modulo $p$. Under this assumption, denoting by $\Omega_n\supseteq \K$ the splitting field of $f_1\circ\dots\circ f_n$ and assuming furthermore that $p$ does not divide the leading coefficient of $f_1\circ\dots\circ f_n$, it follows that $p$ is unramified in $\Omega_n/\K$. Irreducibility modulo $p$ is then equivalent to the Frobenius at $p$ in $\Omega_n/\K$ generating a transitive cyclic subgroup of $\textrm{Gal}(\Omega_n/\K)$. Chebotarev's density theorem tells us that the density of the set of such primes $p$ equals the proportion of elements of $\textrm{Gal}(\Omega_n/\K)$ consisting of a single cycle. 

Furthermore, 
regarding the structure of $\textrm{Gal}(\Omega_n/\K)$, note that as long as $f_1\circ \dots \circ f_n$ is separable, the Galois group $\textrm{Gal}(\Omega_n/\K)$ embeds naturally as a permutation group into the $n$-fold iterated (imprimitive) wreath product $\Gamma_n\wr\dots\wr \Gamma_1$, where $\Gamma_i:=\textrm{Mon}_{f_i}:=\textrm{Gal}(f_i(X)-t/\K(t))$ is the (arithmetic) monodromy group of $f_i$. Here, the imprimitive wreath product $U\wr V$ of transitive groups $U\le S_s$ and $V\le S_r$ is $U\wr V = U^r\rtimes V\le S_{rs}$, acting on $r$ blocks of size $s$ each, with $V$ acting on $U^r$ via permutation of the $r$ copies of $U$.

We therefore need to estimate the proportion of full cycles\footnote{We call an $n$-cycle in a transitive group of permutation degree $n$ a ``full cycle" whenever it is unnecessary or inconvenient to emphasize the permutation degree.} in subgroups of iterated wreath products as above.

\subsection{Prerequisites about primitive and imprimitive permutation groups}
In view of the above, we first need the following classification of primitive permutation groups containing a cyclic transitive subgroup, which rests on the classification of finite simple groups, cf.\ \cite[Theorem 3]{Jones2}. Recall that a finite group $G$ is called almost simple if $S\le G\le Aut(S)$ for a nonabelian simple group $S$.
\begin{theorem}
\label{thm:class}
Let $G\le S_n$ be a primitive permutation group containing an $n$-cycle. Then one of the following holds:
\begin{itemize}
\item[1)] $G$ is solvable, and either\begin{itemize}
\item[1a)] $n=p$ is prime and $C_p\le G\le AGL_1(p)$, where $C_p$ and $AGL_1(p)$ denote the cyclic group of order $p$ and its symmetric normalizer respectively; or
\item[1b)] $G=S_4$.
\end{itemize}
\item[2)] $G$ is nonsolvable, and in fact almost simple with socle $S$ fulfilling one of the following:
\begin{itemize}
\item[2a)] $n\ge 5$ and $S=A_n\le G\le S_n$ (with $n$ odd if $G=A_n$).
\item[2b)] $n=\frac{q^d-1}{q-1}$ for a prime power $q$ and an integer $d\ge 2$ with $(d,q)\notin \{(2,2), (2,3)\}$; $S=PSL_d(q)$ and $PGL_d(q)\le G\le P\Gamma L_d(q)$.
\item[2c)] $n=11$ and $S=G\in \{M_{11}, PSL_2(11)\}$, or $n=23$ and $S=G=M_{23}$.
\end{itemize}
\end{itemize}
\end{theorem}

As noted in \cite[Theorem 1.2(2)]{Praeger}, the assumption of primitivity in Theorem \ref{thm:class} can actually be regained from weaker assumptions.
\begin{theorem}
\label{thm:praeger}
Let $G\le S_n$ be an almost-simple transitive nonsolvable group containing an $n$-cycle. Then $G$ is one of the groups in Case 2) of Theorem \ref{thm:class}; in particular, $G$ is primitive.
\end{theorem}

As a direct consequence, we obtain:
\begin{corollary}
\label{cor:overgr}
Let $G\le S_n$ be an almost simple primitive nonsolvable group containing an $n$-cycle $\tau$, and let $S:=\textrm{soc}(G)$. Then there is no proper cyclic overgroup of $\langle\tau\rangle$ in $Aut(S)$.
\end{corollary}
\begin{proof}
If $\langle\tilde{\tau}\rangle\le Aut(S)$ were a proper cyclic overgroup of $\langle \tau\rangle$, then $\Gamma:=\langle S,\tilde{\tau}\rangle$ would necessarily be a proper overgroup of $\langle S,\tau\rangle$. In particular, $\Gamma$ would act imprimitively on cosets of a point stabilizer in $\langle S,\tilde{\tau}\rangle$. Furthermore, $\langle\tilde{\tau}\rangle$ would act transitively. 
 But by Theorem \ref{thm:praeger}, an imprimitive almost-simple group cannot contain a transitive cyclic subgroup.\end{proof}

We furthermore need the following result on the structure of the block kernel of certain imprimitive groups. 
It is a special case of Lemma 3.1 in \cite{KN}.
\begin{lemma}\label{carmel}
Let $U\le S_s$ and $V\le S_r$ be transitive groups and let $G\leq U\wr V (\le S_{rs})$ be a subgroup whose natural projection to $V$ is onto. Furthermore, denote by $\Delta$ the block of $1$ in the given imprimitive action of $G$, and assume that the image of the block stabilizer $G_{\Delta}$ in its action on $\Delta$ is all of $U$.\footnote{Due to the transitivity of $V$, the images of all block stabilizers on the respective blocks are isomorphic, whence the choice of block here is arbitrary.} 
Assume furthermore that $U$ is almost simple with nonabelian simple socle $L$ and that 
 $K:=G\cap U^r$ is nontrivial.
Then there exists some partition $O_1,\ldots,O_k$ of $\{1,\dots, r\}$ which is preserved by $G$ and such that
$K\cap L^{O_j}\cong L$, $j=1,\ldots,k$ and   $$\soc(K) = K\cap \soc(U)^r  = (K\cap L^{O_1})\times \cdots \times (K\cap L^{O_k}).$$
Here $L^{O_j}$ denotes the subgroup of $x\in L^r$ whose component entries are the identity for all indices outside of $O_j$.
\end{lemma}

\subsection{Monodromy groups of polynomials}
To derive Corollary \ref{cor:manyfibers} from Theorem \ref{thm:main}, we need certain facts about monodromy groups of indecomposable polynomials. To this aim, we recall the full classification of such groups achieved by M\"uller in \cite{Mue}. Even though this will not be used in full here, it should serve as orientation to put the results obtained here into context.

%
%

\begin{theorem}
\label{thm:monclass}
Let $\K\subseteq \mathbb{C}$, and let $f\in \K[X]$ be an indecomposable polynomial\footnote{Differences in wording compared to the Main Theorem of \cite{Mue} are mere technicalities; in particular, it is well known that for the case of polynomials, indecomposability over $\K$ and over $\overline{\K}$ are equivalent, and furthermore the arithmetic monodromy group $A$ is contained in the symmetric normalizer of the geometric monodromy group $\textrm{Gal}(f(X)-t/\mathbb{C}(t))$.} of degree $n$. Let $A:=\textrm{Mon}_f:=\textrm{Gal}(f(X)-t/\K(t))$ be the (arithmetic) monodromy group of $f$. Then one of the following holds:
\begin{itemize}
\item[1)] $C_p\le A\le AGL_1(p)$ for a prime $p$; and in this case, $f$ is linearly related (i.e., equal up to composition with linear polynomials) to $X^p$ or a $p$-th degree Dickson polynomial.
\item[2)] $A\in \{A_n, S_n\}$.
\item[3)] $A$ one of $PGL_2(5)$, $PSL_3(2)$, $PGL_2(7)$, $P\Gamma L_2(8)$, $P\Gamma L_2(9)$, $PSL_2(11)$, $M_{11}$, $PSL_3(3)$, $PSL_4(2)$, $P\Gamma L_3(4)$, $M_{23}$, $PSL_5(2)$, acting of degree $6,7,8,9,10,11,11,13,15,21,23$ and $31$ respectively. For these exceptional cases, all possible polynomials are known explicitly.
\end{itemize}
\end{theorem}

\section{Proof of the main results}

Theorem \ref{thm:main} will be deduced from the following purely group-theoretical theorem, which should also be interesting on its own.
\begin{theorem}
\label{thm:ncycleprop}
Let $G\le S_n$ be a finite transitive permutation group, and let $G=:U_0 > U_1 > \dots > U_m$, $m\in \mathbb{N}$ be a chain of maximal subgroups from $G$ to a point stabilizer $U_m$. For $i=1,\dots, m$, let $\Gamma_i$ be the image of $U_{i-1}$ acting on cosets of $U_i$. Let $\pi(G)$ be the proportion of $n$-cycles among all elements of $G$. Then
$$\pi(G) \le \frac{\varphi(n)}{n\cdot 2^d},$$
where $\varphi$ is the Euler totient function, and $d$ denotes the number of $i\in \{1,\dots, m\}$ for which $\Gamma_i$ does not fulfill $C_p\le \Gamma_i\le AGL_1(p)$ for any prime $p$.
\end{theorem}

Theorem \ref{thm:ncycleprop} improves upon \cite[Theorem 1.5]{Muz}, which established that a finite permutation group of degree $n$ contains at most $\varphi(n)$ conjugacy classes of $n$-cycles, or equivalently $\pi(G)\le \frac{\varphi(n)}{n}$. Our strengthening allows to deduce density-$0$ results in the limit $d\to \infty$ and therefore leads naturally to applications in arithmetic dynamics.

Before proceeding to the proof, we fix some notation. Let $U\le S_r$ be transitive and $H\le S_s$ be primitive. Set $n:=rs$, and let $G\le H\wr U (=H^r\rtimes U) \le S_n$ be a transitive subgroup of the (imprimitive) wreath product. Let $K\trianglelefteq G$ be the kernel under projection to $U$, i.e., $K=G\cap H^r$. Assume that the projection $G\to U$ is onto, and that the image of a block stabilizer in its action on this block is all of $H$. 
Let $\tau\in G$ be an $n$-cycle. Obviously, the image of $\tau$ in $U$ (i.e., the image of $\tau$ in the action on the blocks) must be a transitive cyclic subgroup of $U$, i.e., an $r$-cycle. Hence, in order to estimate the proportion of $n$-cycles in $G$, it is sufficient to estimate the proportion in each coset $K\tau$ of a fixed $n$-cycle $\tau$, and thus assume that $G$ itself is of the form $K\langle\tau\rangle$ (and thus, $U=C_r$).
Denote the proportion of $n$-cycles among all elements of a given subset $M\subset G\le S_n$ by $\pi(M)$. We will show the following, with the emphasis being of course on Conclusion 2):

\begin{proposition}
\label{prop:cruc}
If, with the above notation,
 $\tau\in G$ is an $n$-cycle, then the following hold. \begin{itemize}
  \item[1)] If $s=p$ is a prime and $C_p\le H\le AGL_1(p)$, then\\ $\pi(K\tau)\le \begin{cases}\frac{p-1}{p}, \text{ if } (p,r)=1 \text{ or } $K$ \text{ is not elementary-abelian},\\ 1, \text{ otherwise}\end{cases}$, 
 \item[2)] In any other case, $\pi(K\tau) \le  \frac{1}{2}\cdot \prod_{p|s, (p,r)=1} \frac{p-1}{p}$.
\end{itemize}
\end{proposition}

\begin{proof}[Proof of Theorem \ref{thm:ncycleprop} assuming Proposition \ref{prop:cruc}]
With the notation from Theorem \ref{thm:ncycleprop}, let $U$ be the image of $G$ in the action on cosets of $U_{m-1}$ and $H=\Gamma_m$ the image of $U_{m-1}$ in the action on cosets of $U_m$. Then $G$ embeds into $H\wr U$ by the embedding theorem of Krasner and Kaloujnine, cf.\ \cite{KK}. Furthermore, $H=\Gamma_m$ is indeed a primitive group, since $U_{m}<U_{m-1}$ is a maximal subgroup. Next, let $K$ be the kernel of the projection from $G$ to $U$, and set $r=[G:U_{m-1}]$, $s=[U_{m-1}:U_m]$ and $n=rs$. 
Clearly, the proportion of $n$-cycles in $G$ is upper bounded by the proportion of $r$-cycles in $U$ multiplied by the maximum value $\pi(K\tau)$ where $\tau\in G$ is any given $n$-cycle. 

{\it Case 1}: Assume first that $C_p\le H\le AGL_1(p)$ for some prime $p$. Then inductively (the beginning $r=1$ being trivial), the proportion of $r$-cycles in $U$ is at most $\frac{\varphi(r)}{r\cdot 2^{d}}$, and since $\frac{\varphi(rs)}{rs} = \begin{cases}\frac{p-1}{p}\cdot \frac{\varphi(r)}{r}, \text{ if } (p,r)=1\\ \frac{\varphi(r)}{r}, \text{ otherwise}\end{cases}$, the assertion is immediate from part 1) of Proposition \ref{prop:cruc}.

{\it Case 2}: Now assume $H$ is any other group. Then, by induction and part 2) of Proposition \ref{prop:cruc}, the proportion of $n$-cycles in $G$ is at most $\frac{\varphi(r)}{2^{d-1}r} \cdot \frac{1}{2}\cdot \prod_{p|s, (p,r)=1} \frac{p-1}{p} = \frac{\varphi(rs)}{rs}\cdot \frac{1}{2^d}$.
\end{proof}

\begin{proof}[Proof of Theorem \ref{thm:main} assuming Theorem \ref{thm:ncycleprop}]
Let $F_j:=f_1\circ\dots\circ f_j$, and let $\K\subseteq \K_1\subseteq \K_2\subseteq\dots$ be a chain of fields such that $\K_j$ is generated over $\K$ by a root of $F_j$. Let $G_j:=\textrm{Gal}(F_j/\K)$.  
 Refine the chain $\K\subseteq \K_1\subseteq \K_2\subseteq \dots$ to a chain $\K\subset L_1\subset L_2\subset \dots$ without any proper intermediate fields, let $\Omega_j$ be the Galois closure of $L_j/\K$, and let $\tilde{G}_j = \textrm{Gal}(\Omega_j/\K)$ (viewed in the natural permutation action on cosets of $\textrm{Gal}(\Omega_j/L_j)$). 
Finally, let $\Gamma_j$ denote the Galois group of the Galois closure of $L_j/L_{j-1}$.

Now assume that the proportion of full cycles in $G_j$ (and hence, in $\tilde{G}_j$) does not converge to $0$ as $j\to \infty$. Then by Theorem \ref{thm:ncycleprop}, applied to $G:=\tilde{G}_j$ for each $j$, there are only finitely many indices $i\in \mathbb{N}$ for which $\Gamma_i$ is not a subgroup of any group $AGL_1(p_i)$.  Let $i_0$ be the largest such index, and set $L:=\Omega_{i_0}$. Via the Galois correspondence and the Krasner-Kaloujnine theorem \cite{KK}, the inclusion $\K\subset L_{i_0}\subset L_j$ (for every $j\ge i_0$) yields an embedding of the group $\tilde{G}_j$ as a permutation group into the wreath product $H\wr \tilde{G}_{i_0}$, where $H$ denotes the Galois group of the Galois closure of $L_j/L_{i_0}$. The group $\textrm{Gal}(\Omega_j/L)$ is the kernel of the natural projection $\tilde{G}_j\to \tilde{G}_{i_0}$, and hence embeds into the direct power $H^{[L_{i_0}: \K]}$, which itself embeds (as an intransitive subgroup, with $[L_{i_0}:\K]$ orbits of length $[L_j:L_{i_0}]$) into $H\wr V$ for {\it any} transitive group $V$ of degree $[L_{i_0}: \K]$. Choose, for example, $V$ as a {\it cyclic} transitive group; then both $V$ and $H$ embed as permutation groups into iterated wreath products of suitable groups $AGL_1(p_i)$; for $H$, this is due to the fact that that the $\Gamma_i$, $i> i_0$ all embed into such a group.
In particular, it follows that every $\textrm{Gal}(F_j/L)$, $j\in \mathbb{N}$ embeds into an iterated wreath product of  groups $AGL_1(p_i)$, completing the proof. 
\end{proof}

We take some time to include a minor addition to the conclusion of Theorem \ref{thm:main}, showing that in the case where all $f_i$ are of bounded degree, the orders of the Galois groups of iterates $f_1\circ\dots\circ f_n$ are actually upper bounded by iterated wreath products of cyclic, not only of affine groups.
\begin{theorem}
\label{thm:main_add}
Let $\mathcal{S}:=(f_n)_{n\in \mathbb{N}}$ a sequence of polynomials of bounded degrees $d_i\le N$ ($i\in \mathbb{N}$) over a number field $\K$, and assume that the set of stable primes of $\mathcal{S}$ is of positive density, and let $G_n:=\textrm{Gal}(f_1\circ\dots\circ f_n/\K)$. Then $|G_n| \ll_n c(d_1\cdots d_n)$, where $c(n)$ denotes the maximal order of a transitive iterated wreath product of cyclic groups inside $S_n$.
\end{theorem}

\begin{remark}
An easy induction shows $c(n)\le 2^{n-1}$, with equality reached via the iterated wreath product of groups $C_2$ in the case where $n$ is a $2$-power.
\end{remark}
\begin{proof}[Proof of Theorem \ref{thm:main_add} assuming Proposition \ref{prop:cruc}]
For this, note that in Case 1 of the proof of Theorem \ref{thm:ncycleprop}, 
 we have made use of the bound
$\pi(K\tau)\le \frac{p-1}{p}$ from part 1) of Proposition \ref{prop:cruc} {\it only} in the case where $(p,r)=1$, and {\it not} in the case where $p$ divides $r$ and $K$ is not elementary-abelian. \footnote{Of course, a corresponding strengthening could be included in the wording of Theorem \ref{thm:ncycleprop}; we decided against this in favor of readability of the statement, and instead present the adaptation ad-hoc.} Note that under the assumption that all $f_i$ are of bounded degree, an infinite product of such factors $\frac{p-1}{p}$ would necessarily be zero. In order for the density of stable primes to be positive, it is therefore necessary that this case occurs only finitely many times along the chain $\K\subset L_1\subset L_2\subset \dots$ of intermediate fields as in the previous proof.
On the other hand, we already know that (in the notation of the previous proof) all but finitely many of the groups $\Gamma_j$ must embed into some $AGL_1(p_j)$; furthermore, $p_j$ necessarily divides $[L_{j-1}:\K]$ for all but finitely many $j$, again simply by finiteness of the set of available primes $p_j$. In total, we have that, in order for the set of stable primes to be of positive density, all but finitely $j\in \mathbb{N}$ must fulfill $C_{p_j}\le \Gamma_j\le AGL_1(p_j)$ \textit{and} $K_j:=\ker(\tilde{G}_j\rightarrow \tilde{G}_{j-1})$ is elementary-abelian (necessarily of exponent $p_j$). 
This shows that, up to a constant factor, the orders of all $\tilde{G}_n$ are bounded from above by the order of some iterated wreath product of cyclic groups inside the respective symmetric groups.
\end{proof}

We next turn to the proof of the crucial Proposition \ref{prop:cruc}. 
\begin{proof}[Proof of Proposition \ref{prop:cruc}]
{\it Part 1: General preparations:} 
Throughout, we take the convention of multiplying permutations from left to right. 
As above, let $G\le H\wr C_r \le S_{rs}$ with $H\le S_s$ primitive. Assume that $\tau \in G$ is an $n$-cycle ($n:=rs$). 
We denote the $r$ blocks of the imprimitive action of $G$ by $\Delta_1,\dots, \Delta_r$. Upon relabelling elements, we may and will assume $\Delta_i:=\{i+k\cdot r\mid k\in \{0,\dots s-1\}\}$ for each $i=1,\dots, r$. We additionally fix an ordering of the elements of each block from smallest to largest, so that later, when focussing on the action of the blocks kernel on on the $i$-th block, we may, e.g., identify the permutation $(i, i+r, \dots, i+(s-1)r)$ with the $s$-cycle $\rho:=(1,2,\dots, s)\in S_s$.
Furthermore, we assume that the given $n$-cycle $\tau$ is of the form $\tau= (1,2,\dots, n)$. This assumption is without loss, up to conjugation by a suitable permutation preserving the block system and thus in particular inducing an automorphism of the block kernel $K\triangleleft G$.

We take a moment to consider an important detail of the further strategy. In order to count $n$-cycles in $K\tau$, we will often argue iteratively using certain subgroups $K_0$ of $K$, i.e., first  estimate the number of cosets of $K_0$ contained in $K\tau$ which contain any $n$-cycle $\tau'$ at all, followed by estimating the proportion of $n$-cycles in this coset $K_0\tau'$. For sake of simplicity, we will, in the second part of such an argument, again assume the newly identified $n$-cycle $\tau'$ to be of the form $(1,2,\dots, n)$ chosen above for $\tau$. There will be no risk of ambiguity here, but it is important to note that this change always corresponds to applying an automorphism of $K$, and hence, in order to be sure that elements of $K\tau$ are counted exactly once in total, we have to be sure that $K_0$ is also invariant under this automorphism.

Proceeding with the proof, we now consider the block kernel $K$. Clearly $K$  is nontrivial, since indeed $\tau^r\in K$. 
We first claim that, for $\sigma = (\sigma_1,\dots, \sigma_r)\in K$, the element $\sigma\tau$ is an $n$-cycle if and only if 
$\sigma_1\cdots\sigma_r\cdot \rho\in S_s$ is an $s$-cycle. 

Indeed,  since $\langle\sigma\tau\rangle$ is transitive on the blocks, it is clear that $\sigma\tau$ is an $n$-cycle if and only if $(\sigma\tau)^r\in K$ acts as an $s$-cycle on one -- and hence on every -- block. But $(\sigma\tau)^r = \sigma \cdot \sigma^{\tau^{-1}}\cdot \sigma^{\tau^{-2}}\cdots \sigma^{\tau^{1-r}} \cdot \tau^r$, where $\sigma^x = x^{-1}\sigma x$ denotes the conjugate of $\sigma$ by $x$. Now the first entry of $\sigma^{\tau^{-k}}$ equals $\sigma_{1+k}$. Furthermore, the first entry of $\tau^r$  is simply $\rho$ by definition. In total, the first entry $[(\sigma\tau)^r]_1$ of $(\sigma\tau)^r$  equals 
\begin{equation}
\label{eq:product}[(\sigma\tau)^r]_1 = \sigma_1 \sigma_2\cdots\sigma_r\cdot \rho,
\end{equation} 
proving the claim.

{\it Part 2: The case of nonsolvable $H$}: 
 We now delay the case of solvable $H$ until later and assume for the moment that $H$ is nonsolvable. In particular, $H$ necessarily contains an $s$-cycle and is thus one of the groups in case 2) of Theorem \ref{thm:class}, and in particular is an almost simple group. Denote its socle by $S$. 
 
We now proceed to estimating the proportion of $n$-cycles in $K\tau$. 
As indicated above, we do this in two steps: 
upper-bounding (say, by a value $\alpha$) the proportion of cosets $K_0x$ of $K_0$ in $K$ for which $K_0x\tau$ contains any $n$-cycle $\tau'$ at all, where $K_0:=\soc(K)$; followed by upper-bounding (say, by a value $\beta$) the proportion of $n$-cycles in $K_0\tau'$. The total proportion of $n$-cycles in $K\tau$ is then upper bounded by $\alpha\beta$.

{\it Part 2a: The ``$\alpha$-estimate"}:
We will aim at the following upper bounds $\alpha$ for the proportion of cosets $K_0x$ of $K_0$ in $K$ for which $K_0 x\tau$ contains any $n$-cycle.
\begin{equation}
\label{eq:alpha}
\alpha = \begin{cases}
\frac{1}{2}, \text{ if } S=A_s \text{ with } s \text{ even and } r \text{ odd},
\\
\prod_{p|(d,q-1),  \ (p,r)=1} \frac{p-1}{p}, \text{ if } S=PSL_d(q),\\
1, \text{ otherwise.}
\end{cases}
\end{equation}
Of course for cases with $\alpha=1$, nothing needs to be proven. Thus, assume first that $S=A_s$ with $s$ even, and that $r$ is odd. In this case, $\tau^r$ is a disjoint product of an odd number of $s$-cycles in $K$. In particular, the homomorphism $\psi: K\to C_2$, $\sigma:=(\sigma_1,\dots, \sigma_r)\mapsto \textrm{sgn}(\sigma_1\cdots \sigma_r)$ is onto. On the other hand, for $\sigma\tau$ to be an $n$-cycle, it is necessary by Equation \eqref{eq:product} that $\sigma_1\cdots \sigma_r\cdot \rho$ is an $s$-cycle, and thus in particular that $\psi(\sigma) = +1$. This gives the upper bound $\alpha=\frac{1}{2}$ as claimed.

Assume next that $S=PSL_d(q)$ with $d\ge 2$. Set $K_1:=K\cap PGL_d(q)$. Note that $K_1/K_0$ is abelian since $PGL_d(q)/PSL_d(q)$ is cyclic (of order $(d,q-1)$). Here, we will reach our $\alpha$-value by counting over cosets $K_0x$ of $K_0$ inside $K_1\tau$.\footnote{I.e, we silently replace $\tau$ with an $n$-cycle in some given coset of $K_1$, before counting $n$-cycles in the individual $K_0$-cosets of this $K_1$-coset. This is fine since $K_1$ is a characteristic subgroup of $K$ due to $PGL_d(q)$ being characteristic in $P\Gamma L_d(q)$.}
 Set $\nu:=\prod_{p|(d,q-1),  \ (p,r)=1} \frac{p-1}{p}$, and note that, since $PGL_d(q)/PSL_d(q)$ is cyclic of order $(d,q-1)$, it has a unique quotient isomorphic to $C_{\nu}$.
 Consider the homomorphism $\psi:K_1\to C_{\nu}$ given by $(\sigma_1,\dots, \sigma_r)\mapsto \overline{\sigma_1\cdots \sigma_r}$, with the bar denoting the image in the quotient $C_{\nu}$. We may assume $(d,q)\ne (2,8)$, since for that case we only claim the trivial bound $\alpha=1$ anyway.  Then the only $s$-cycles in $P\Gamma L_d(q)$ are the Singer cycles of $PGL_d(q)$ (cf.\ \cite[Theorem 1]{Jones2}), and those in particular all generate the full quotient $PGL_d(q)/PSL_d(q)$ (of order $(d,q-1)$). In particular, we have $\tau^r\in K_1$, and furthermore $\psi(\tau^r)=\overline{\rho}^r$, which generates $C_{\nu}$ due to $(r,\nu)=1$. I.e., $\psi$ is surjective. But now, using again Equation \eqref{eq:product}, we see that for $\sigma\tau$ to be an $n$-cycle, it is necessary that $\psi(\sigma)\cdot \overline{\rho}$ generates $C_{\nu}$, which is possible only for $\varphi(\nu)$ values of $\psi(\sigma)$. Hence, we get the upper bound $\alpha=\frac{\varphi(\nu)}{\nu} = \prod_{p|(d,q-1),  \ (p,r)=1} \frac{p-1}{p}$, as claimed. 

{\it Part 2b: The ``$\beta$-estimate"}:
 Next, we bound the proportion of $n$-cycles in $K_0\tau'$, for a fixed $n$-cycle $\tau'$.  We claim that this proportion is bounded from above by a value $\beta$ as follows.
 \begin{equation}
 \label{eq:beta}
 \beta = \begin{cases}
 \frac{2}{s}, \text{ if } S=A_s, \text{ or } S \text{ is as in Case 2c) of Theorem \ref{thm:class}},\\
 \dfrac{\frac{\varphi(s)}{s}}{\frac{\varphi((d,q-1))}{(d,q-1)}}\cdot \frac{1}{d}, \text{ if } S=PSL_d(q) (\text{and } s=\frac{q^d-1}{q-1}).
 \end{cases}
 \end{equation}
 Since $K_0$ is a characteristic subgroup of $K$, we can without loss assume $\tau'=(1,2,\dots, n)$ again, i.e., continue working with $\tau$ instead of $\tau'$, as explained in Part 1. 
 Note that $K_0=\soc(K)$ is nontrivial, since $K$ is. By Lemma \ref{carmel}, 
we thus have $K_0\cong S^{r'}$ for some divisor $r'$ of $r$, and furthermore, there exists a partition of $\{1,\dots, r\}$ into $r'$ sets $M_1,\dots, M_{r'}$ of equal size $\frac{r}{r'}$, preserved by the blocks action of $G$, and such that for each of the $r'$ simple normal subgroups $S_{(i)} \cong S$ of $K_0$ ($i=1,\dots, r'$) the following holds: 
$$S_{(i)} = \textrm{soc}(K) \cap S^{M_i},$$
where $S^{M_i}$ denotes the subgroup of all $x\in S^r$ whose component entries are the identity for all indices outside of $M_i$. 
Due to the action of $\tau$, necessarily $M_i = \{i+kr' \mid k\in \{0, \frac{r}{r'}-1\}\}$. In particular, since $S_{(i)}$ is a simple group surjecting on each component with index contained in $M_i$, entries of an arbitrary element $\sigma:=(\sigma_1,\dots, \sigma_r)\in \soc(K) \le S^r$ in those components differ only by an automorphism of $S$, i.e., for all $i=1,\dots, r$, there exists $\rho_{i}\in Aut(S)$ such that $\sigma_{i+r'} = \sigma_i^{\rho_{i}}$ (with the addition in the index of course to be understood modulo $r$).

We can be more precise about these automorphisms by considering the action of $\tau$. 
Indeed, note that for $\sigma=(\sigma_1,\dots, \sigma_r)\in K$, we have $\sigma^\tau = (\sigma_r^\rho, \sigma_1, \sigma_2, \dots \sigma_{r-1})$, where $\rho:=(1,\dots, s)\in S_s$ (the occurrence of $\rho$ is due to the fact that $\tau$ maps the $j$-th element of the $i$-th block to the $j$-th element of the $i+1$-th block for $i=1,\dots, r-1$, but maps the $j$-th element of the last block to the $\rho(j)$-th element of the first block).
On the other hand, for each $i$, there exists a unique automorphism $\rho_i\in Aut(S)$ such that $x_{i+r'} = x_i^{\rho_i}$ for {\it all} $x\in \soc(K)$, and we may apply this in particular for $x=\sigma^{\tau}$. For example, since the entries in components number $2$ and  $r'+2$ respectively of this element are $\sigma_1$ and $\sigma_{r'+1}=\sigma_1^{\rho_1}$, we find that $\rho_2=\rho_1$, and in the same way 
\begin{equation}
\label{eq:cases}\rho_i = \begin{cases}\rho_1, \text{ for } i\in \{1,\dots, r-r'\}\\
\rho\rho_1, \text{ for } i\in \{r-r'+1, \dots, r\}\end{cases}.\end{equation}
Now choose $\sigma\in S_{(i)}$ (i.e., $\sigma_j=1$ for all $j$ not congruent to $i$ modulo $r'$) for some fixed $i\in \{1,\dots, r'\}$. The nonidentity component entries of $\sigma$ then fulfill the relations $\sigma_{i+r'}=\sigma_i^{\rho_1}$, $\sigma_{i+2r'}=\sigma_{i+r'}^{\rho_1} = \sigma_i^{\rho_1^2}$, $\dots$, and finally $\sigma_i = \sigma_{i+(\frac{r}{r'}-1)r'}^{\rho\rho_1} = \dots = \sigma_i^{\rho\cdot \rho_1^{r/r'}}$. This means that ${\rho\cdot \rho_1^{r/r'}}$ centralizes $S$ and hence must be the identity. In other words, \begin{equation}
\label{eq:rho}
\rho_1^{r/r'}=\rho^{-1}.\end{equation} 
So $\langle\rho_1\rangle$ is a cyclic overgroup of $\langle\rho\rangle$ inside $Aut(S)$. By Corollary \ref{cor:overgr}, 
 $\rho_1$ must be a generator of the cyclic group $\langle\rho\rangle \le Aut(S)\cap S_s$.
 Incidentally, this also implies that $r/r'$ must be coprime to $s$.

Now to bound the proportion of $n$-cycles in $K_0\tau$ for a fixed $n$-cycle $\tau$, we may in fact instead estimate the proportion of $n$-cycles in $S_{(i)}\tau$ for any fixed $i$, say $i=1$ for convenience.
\footnote{We again point out a noteworthy detail: Counting by cosets of $S_{(i)}$ requires first choosing an $n$-cycle $\tilde{\tau}$ in any fixed coset of $S_{(i)}$ inside $K_0\tau$. Again, we will pretend $\tilde{\tau}$ to be of the ``default form" $(1,2,\dots, n)$ chosen previously for $\tau$, and again this corresponds to performing an automorphism of $K$. The subgroups $S_{(i)}$ themselves are {\bf not} necessarily fixed by this automorphism; however, they are necessarily permuted among each other, due to being the unique minimal normal subgroups of the block kernel $K$. The fact that the following estimates work in the same way for {\it each} $S_{(i)}$ means that we will indeed obtain a correct estimate for the proportion of $n$-cycles in all of $K_0\tau$ in this way.}

Assume that $\sigma\in S_{(1)}$ is such that $\sigma\tau$ is an $n$-cycle. In particular, $(\sigma\tau)^r\in K$ has an $s$-cycle in each of its components.

But by \eqref{eq:product}, the first entry of $(\sigma\tau)^r$ equals 
$\sigma_1 \sigma_1^{\rho_1}\cdots \sigma_1^{\rho_1^{r/r'-1}}\cdot \rho = (\sigma_1\rho_1^{-1})^{r/r'}\cdot \rho_1^{r/r'} \cdot \rho$, and this equals  $(\sigma_1\rho_1^{-1})^{r/r'}$ by \eqref{eq:rho}.

Here, $\rho_1\in Aut(S)$ is fixed, and we are therefore reduced to estimating how often $\sigma_1\rho_{1}^{-1}$ can be an $s$-cycle as $\sigma_1$ moves through  $S$. This amounts to estimating the highest number of $s$-cycles occurring in any coset of $S$ inside its symmetric normalizer $Aut(S)\cap S_s$. We can do this by going through the list of cases in Theorem \ref{thm:class}.

For $S=A_s$ the alternating group ($s\ge 5$), clearly the number of $s$-cycles in a coset of $S$ is at most $\frac{2}{s}|S|$. This is because the total number of $s$-cycles in $S_s$ is $(s-1)! = |S|\cdot \frac{2}{s}$. Hence the proportion of $n$-cycles in $K_0\tau$ (with $K_0:=K\cap (A_s)^r$) is at most $\frac{2}{s}$. 

Next, assume $S=PSL_d(q)$ with $d\ge 2$ and $(d,q)\ne (2,8)$. As seen already, the only $s$-cycles in $P\Gamma L_d(q)$ are then the Singer cycles of $PGL_d(q)$. It is known that all cyclic subgroups generated by Singer cycles are conjugate to each other in $PGL_d(q)$ (e.g., \cite[Corollary 2]{Jones2}), and furthermore for any Singer cycle $x$, one has $|N_{PGL_d(q)}(\langle x\rangle)/C_{PGL_d(q)}(x)| = d$, cf.\  \cite[Chapter II, Satz 7.3]{Huppert}. In particular, exactly $d$ of the $\varphi(s)$ generators of one given cyclic Singer subgroup are conjugate to each other in $PGL_d(q)$. Also, since $PGL_d(q)/PSL_d(q)\cong C_{(d,q-1)}$ is cyclic, conjugate elements necessarily lie in the same coset of (the commutator subgroup) $PSL_d(q)$. From this information, we can obtain a formula for the proportion of full cycles in a given coset of $PSL_d(q)$ inside $PGL_d(q)$. Indeed, there are a total of $\varphi((d,q-1))$ cosets containing $s$-cycles at all; the total number of conjugacy classes of $s$-cycles is $\frac{\varphi(s)}{d}$, whence the total number of $s$-cycles is $|PGL_d(q)|/s \cdot \frac{\varphi(s)}{d}$. The number of $s$-cycles in each coset (containing any such cycle) is thus 
$$|PGL_d(q)|\cdot \frac{\varphi(s)}{d\cdot s\cdot \varphi((d,q-1))} = |PSL_d(q)|\cdot \dfrac{\frac{\varphi(s)}{s}}{\frac{\varphi((d,q-1))}{(d,q-1)}}\cdot \frac{1}{d},$$ yielding the claimed value for $\beta$.
%

The remaining cases $S\in \{PSL_2(8), PSL_2(11), M_{11}, M_{23}\}$ (of degree $9,11,11$ and $23$ respectively) can quickly be checked by hand. In fact, the last three simple groups are all their own symmetric normalizer, whence it suffices to check the proportion of full cycles in the simple group itself. This proportion is $\frac{2}{11}$, $\frac{2}{11}$ and $\frac{2}{23}$ respectively. Finally, let $S=PSL_2(8) < P\Gamma L_2(8)$. Here each of the three cosets of the socle contains a $9$-cycle proportion of exactly $\frac{1}{3} = \frac{\varphi(9)}{9}\cdot \frac{1}{2}$.


%

%
%
{\it Part 3: The case $H$ solvable}: 
We finally treat the case  of solvable $H$, i.e., $C_p\le H\le AGL_1(p)$ for some prime $p$; or $H\cong S_4$. Here, we cannot apply Lemma \ref{carmel}. 

{\it Part 3a: The case $H=S_4$}:
First, consider the (more complicated) case $H=S_4$. Since the image of $K$ in the action on any given block is a normal subgroup of $H$ containing a $4$-cycle, we know that this image is still all of $S_4$. Consider now the characteristic subgroup $K_0:= K\cap A_4^r$ of $K$. Firstly, the ``$\alpha$-value" $\alpha=\begin{cases}\frac{1}{2}, \text{ if } s \text{ even and } r \text{ odd},\\ 1, \text{ otherwise}\end{cases}$ is obtained using this subgroup $K_0$ just as before. 
To obtain the $\beta$-value $\beta=\frac{1}{2}$ (and hence, once again an upper bound $\alpha\beta$ as claimed in part 2) of the assertion of Proposition \ref{prop:cruc}), we replace $G$ by $G_0:=K_0\langle\tau\rangle$ for the moment. This group still has block kernel projecting onto $S_4$ in every component due to the presence of a $4$-cycle, Furthermore, since $K_0$ is of index $2$ in this block kernel, it follows that projection of $K_0$ to any component yields (a normal subgroup of $A_4$ of index at most $2$, i.e.,) all of $A_4$. Moreover, the $2$-Sylow subgroup $V$ of $K_0$ is an abelian normal subgroup of $G_0$, and $K_0$ is {\bf $2$-perfect}, meaning that there is no index-$2$ normal subgroup in $K_0$; indeed, if there were such a subgroup $W$ of $K_0$, then by Maschke's theorem in representation theory (applied to the linear action of a $3$-Sylow subgroup of $K_0$ on the $2$-Sylow subgroup), it would be complemented in $K_0$ by another normal subgroup $W'\le V$, and then $W'\subseteq Z(K_0)$ - but clearly $Z(K_0) = \{1\}$ since even the component projections of $K_0$ have trivial center.
We can thus apply a theorem of Higman \cite{Higman} to obtain that $V$ has a complement in $G_0$, i.e., $G_0=VX$ for some subgroup $X$ with $X\cap V=\{1\}$. 
Choose any element of $X$ in the same $V$-coset as $\tau$. This element cannot be a $4r$-cycle, since indeed the $2r$-th power of any $4r$-cycle in $G$ lies in $V=G_0\cap V_4^r$, and $V\cap X=\{1\}$. Therefore, the proportion of full cycles in $K_0\tau$ is less than $1$.

Since any conjugacy class of $4r$-cycles in $G_0$ is of length $\frac{|G_0|}{4r}$ (due to the fact that a transitive cyclic subgroup is its own centralizer) and furthermore the conjugacy class of such an element $\tau$ is contained in $K_0\tau$ \footnote{Indeed, this is due to $G_0/K_0 \cong C_{2r}$ being abelian.} (which has cardinality $\frac{|G|}{2r}$), we have obtained that the proportion of full cycles in $K_0\tau$ is in fact at most $\beta=\frac{1}{2}$. 

{\it Part 3b: The case $H\le AGL_1(p)$}: Set $K_0:=K\cap C_p^r$. First assume that $p$ and $r$ are coprime. Then the normal subgroup $K_0$ has a complement in $\langle K_0,\tau\rangle$ by the Schur-Zassenhaus theorem. In particular, not every element in $K_0\tau$ is an $n$-cycle, whence $\pi(K_0\tau) < 1$. But just as in the case $H=S_4$, since conjugacy classes of $n$-cycles in $\langle K_0,\tau\rangle$ respect $K_0$-cosets, there can be at most $p-1$ such conjugacy classes per coset, yleiding $\pi(K_0\tau)\le \frac{p-1}{p}=:\beta$ under the assumption $(p,r)=1$. Next, assume that $p$ divides $r$, but $K$ is not an elementary-abelian $p$-group. In this case, we can argue just as in Case 3a) to see that $K\tau$ cannot consist only of $n$-cycles (and hence, that the bound $\beta=\frac{p-1}{p}$ is valid here as well): indeed, $K$ is then $p$-perfect, and by applying Higman's theorem \cite{Higman}, we see that $K_0$ has a complement in $K\langle\tau\rangle$, which as above yields elements in $K\tau$ which are not full cycles.

{\it Part 4: Final calculations}:
It remains to verify that the product $\alpha\beta$ of upper bounds obtained above is in all cases less or equal to the bound claimed in Proposition \ref{prop:cruc}. 
For $C_p\le H\le AGL_1(p)$, there is nothing left to show since the $\beta$-value obtained above already gives the bound claimed in case 1) of the proposition. 
For $A_s\le H\le S_s$ ($s\ge 5$) odd, we have the value $\alpha\beta = \frac{2}{s}=\frac{1}{2}\cdot\frac{4}{s}\le \frac{1}{2}\frac{\varphi(s)}{s}\le \frac{1}{2} \prod_{p|s, (p,r)=1} \frac{p-1}{p}$. Similarly, for $H=S_s$ ($s\ge 4$ even), we have $\alpha\beta=\frac{2}{s}$ or $\frac{1}{s}$, depending on whether $r$ is or is not divisible by $2$, and this is once again upper bounded by $\frac{1}{2}\prod_{p|s, (p,r)=1} \frac{p-1}{p}$.

For $S=PSL_d(q)$ (and $s=\frac{q^d-1}{q-1}$, we have $\alpha\beta=(\prod_{p|(d,q-1),  \ (p,r)=1} \frac{p-1}{p}) \dfrac{\frac{\varphi(s)}{s}}{\frac{\varphi((d,q-1))}{(d,q-1)}}\cdot \frac{1}{d} =\frac{1}{d}\cdot  \dfrac{\prod_{p|s}\frac{p-1}{p}}{\prod_{p|(d,q-1), p|r} \frac{p-1}{p}} \le \frac{1}{d} \prod_{p|s, (p,r)=1}\frac{p-1}{p}$, clearly satisfying the claimed bound. 
Lastly, the exceptional cases $M_{11}$ and $PSL_2(11)$ of degree $s=11$ and $M_{23}$ of degree $s=23$ all give a bound of $\alpha\beta=\beta=\frac{2}{s}$, and the claimed bound is obvious from this, completing the proof.
\end{proof}

We next deduce Corollary \ref{cor:manyfibers} from Theorem \ref{thm:ncycleprop}.
\begin{proof}[Proof of Corollary \ref{cor:manyfibers}]
Assume that in the sequence $(f_n)_{n\in \mathbb{N}}$, there are infinitely many elements not linearly related over $\K$ to a monomial or Dickson polynomial. Since those polynomials are indecomposable, it follows  that they must have monodromy group $S_4$ or a nonsolvable group, cf.\ Theorem \ref{thm:monclass}.
For given $\epsilon>0$, choose $d\in \mathbb{N}$ such that $2^{-d}<\epsilon$, and choose $n\in \mathbb{N}$ such that at least $d$ polynomials out of $f_1,\dots, f_n$ are not linearly related to a monomial or Dickson polynomial.
Let $G_n:=\textrm{Gal}((f_1\circ \dots\circ f_n)(X)-t/\K(t))$ be the monodromy group of the $n$-th compositum $f_1\circ \dots\circ f_n$. Due to Hilbert's irreducibility theorem, for all $a\in \K$ outside of some thin set, one has $\textrm{Gal}((f_1\circ\dots\circ f_n)(X)-a/\K) \cong G_n$. I.e., for such $a$, in the tower $\K\subset \K(f_1^{-1}(a))\subset \K(f_2^{-1}(f_1^{-1}(a)))\subset \dots\subset \K((f_1\circ\dots\circ f_n)^{-1}(a))$, there are at least $d$ (indecomposable) steps whose Galois group is $S_4$ or nonsolvable.
Consequently, due to Theorem \ref{thm:ncycleprop}, the proportion of full cycles in $G_n$, and with it the proportion of stable primes for $(f_1-a, f_2,f_3,\dots)$, is upper bounded by $2^{-d}<\epsilon$.
\end{proof}

\section{A conditional result}
Finally, we will address Question \ref{ques:main}. 
The following question is asked in a much broader context in \cite{BDGHT}. We only extract the special case of polynomial maps over number fields relevant for our application. Other special cases, such as the case of quadratic rational functions, reduce to previously made conjectures, see \cite[Conj.\ 3.11]{Jones13}.
\begin{question}
\label{ques:bridy}
Let $\K$ be a number field and $f\in \K[X]$ be a polynomial of degree $\ge 2$. Let $G_{t,\infty}:=\varprojlim Gal(f^{\circ n}(X)-t/\K(t))$ (with $t$ transcendental), and for $a\in \K$, let $G_{a,\infty}:=\varprojlim \textrm{Gal}(f^{\circ n}(X)-a/\K)$. Is it true that $[G_{t,\infty}:G_{a,\infty}]$ is finite whenever none of the following hold?
\begin{itemize}
\item[1)] $a$ lies in the forward orbit of some ramification point of $f$.
\item[2)] $a$ 
is fixed by some nonidentity map $\psi: \mathbb{P}^1_\K\to \mathbb{P}^1_\K$ fulfilling $\psi\circ f^{\circ n} = f^{\circ n}\circ \psi$ for some $n\ge 1$.
\end{itemize}
\end{question}

Note that Case 2) above includes the case of $a$ being a periodic point, since one may then choose $\psi$ as a suitable iterate of $f$ itself.

\begin{theorem}
\label{thm:evidence}
Assuming that Question \ref{ques:bridy} has a positive answer, it follows that for any polynomial $f\in \K[X]$ which is not a composition of monomials, Dickson polynomials and linear polynomials, and for all $a\in \K$, the set of stable primes of $(f,a)$ is of density $0$.
\end{theorem}
\begin{proof}
Note first that a point $a$ lying in the forward orbit of some ramification point of $f$ leads to $f^{\circ n}(X)-a$ being inseparable for some, and hence all sufficiently large $n$, and hence there are no stable primes at all for $(f,a)$ in this case. Regarding Condition 2) in Question \ref{ques:bridy}, we divide the investigation into the case $\deg(\psi)=1$ and $\deg(\psi)>1$. As for the latter case, it follows from work of Ritt (\cite{Ritt}; see also Theorem 1.2 and the following discussion in \cite{Pakovich}) that a polynomial (here, $f^{\circ n}$) commuting with some rational function $\psi$ of degree $>1$ must either be linearly related to a monomial or Dickson polynomial, or must have a common iterate with $\psi$.
%
%
 Since we have excluded the former, it suffices to consider the case were $f^{\circ N} = \psi^{\circ M}$ for suitable $M,N\in \mathbb{N}$. But then Condition 2) implies that $a$ is invariant under $f^{\circ N}$, and in particular $a$ is a periodic point of $f$. In this case as well, it is obvious that $(f,a)$ cannot have any stable primes. 
 
 Thus, still excluding the case $\deg(\psi)=1$ in Condition 2), existence of stable primes for $(f,a)$ under our assumptions automatically implies finiteness of $[G_{t,\infty}:G_{a,\infty}]$, and then the set of stable primes is of density $0$ by Theorem \ref{thm:main}, since by assumption, $G_{t,\infty}$ does not embed into an iterated wreath product of groups $AGL_1(p_i)$ even after finite base change. 
 
  Lastly, assume we are in Condition 2), but with $\deg(\psi)=1$ (i.e., $\psi$ is a non-identity M\"obius transformation commuting with some iterate $f^{\circ n}$ of $f$ 
  and fixing $a$). We may also assume $\{f^{-n}(a): n\in \mathbb{N}\}$ to be infinite, since the case of periodic points has already been dealt with. Note that for $f$ not linearly related to a monomial, the group of all M\"obius transformations (even over $\mathbb{C}$) commuting with some iterate of $f$ is a finite group, cf.\ \cite[Theorem 1.2]{Pakovich2}. Since a non-identity M\"obius transformation  does not fix more than two points, there exists $n\in \mathbb{N}$ such that $b:=f^{-n}(a)$ is not fixed by any non-identity M\"obius transformation commuting with some iterate of $f$. Note also that a stable prime (of $\K$) for $(f,a)$ is necessarily extended by a stable prime (of $\K(b)$) for $(f,b)$. But then, due to the above treatment, we may assume that $(f,b)$ fulfills neither Condition 1) nor 2) of Question \ref{ques:bridy}, implying as above that $G_{b,\infty}$ does not embed into an iterated wreath product of groups $AGL_1(p_i)$ after finite base change. The same then holds a fortiori for $G_{a,\infty}$, concluding the proof.
\end{proof}

 \end{document}